\documentclass[11pt]{amsart}
\usepackage{mathtools}
\usepackage{amssymb}
\usepackage[margin=1in]{geometry}
\usepackage[utf8]{inputenc}
\usepackage{mathtools}
\usepackage[pdfdisplaydoctitle=true,
            colorlinks=true,
            urlcolor=blue,
            citecolor=blue,
            linkcolor=blue,
            pdfstartview=FitH,
            pdfpagemode= UseNone,
            bookmarksnumbered=true]{hyperref}
\newtheorem{theorem}{Theorem}
\newtheorem{corollary}[theorem]{Corollary}
\newtheorem{lemma}[theorem]{Lemma}
\newtheorem{proposition}[theorem]{Proposition}
\theoremstyle{definition}
\newtheorem{definition}[theorem]{Definition}

\theoremstyle{remark}
\newtheorem{remark}[theorem]{Remark}

\newcommand{\RR}{\mathbb{R}}
\newcommand{\ZZ}{\mathbb{Z}}

\newcommand{\norm}[1]{\left\Vert#1\right\Vert}
\newcommand{\abs}[1]{\left\vert#1\right\vert}

\newcommand{\id}{\mathrm{Id}}

\mathtoolsset{showonlyrefs}
\begin{document}

\title{Local Well-posedness of the Camassa-Holm equation on the real line}

\author[J. Lee]{Jae Min Lee}
\address{Department of Mathematics, The Graduate Center, City University of New York, NY 11106, USA}
\email{jlee10@gradcenter.cuny.edu}

\author[S.C. Preston]{Stephen C. Preston}
\address{Department of Mathematics, Brooklyn College, Brooklyn, NY 11210, USA and Department of Mathematics, The Graduate Center, City University of New York, NY 10016, USA}
\email{stephen.preston@brooklyn.cuny.edu}

\subjclass[2010]{35Q35, 53D25}%
\keywords{Camassa-Holm equation, local existence.}%

\date{\today}%

\maketitle

\section{Introduction}

We study the one dimensional nonperiodic Camassa-Holm [CH] equation
\begin{equation}\label{CH1}
u_t-u_{txx}+3uu_x-uu_{xxx}-2u_x u_{xx}=0,\;\;x,t \in \RR
\end{equation}
It was originally proposed by Camassa and Holm \cite{CH1993, CH1994} as a model for shallow water waves. The equation has remarkable properties like infinitely many conserved quantities, soliton-like solutions, and bi-hamiltonian structures. The solutions to the equation can also be interpreted as geodesics of the right invariant Sobolev $H^1$ metric on the diffeomorphism group of the circle. Background and details on these topics can be found in the following papers and references in these works: for symmetries, complete integrability, and bi-hamiltonian structures, see ~\cite{CH1993, FF1981, AK1999}; for the geometric interpretation, see \cite{Mis1998, KM2003, EM1970}; for peakon and soliton solutions, see ~\cite{BSS2000, ConM1999, Len2005}.

The local well-posedness of the corresponding Cauchy problem of the CH equation in both periodic and nonperiodic cases has been studied extensively. In 1997, Constantin \cite{Con} showed the local well-posedness in the Sobolev spaces $H^s(S^1)$ for $s \ge 4$ where $S^1=\RR/\ZZ$. Then Constantin and Escher \cite{ConE1} improved the result in 1998 with $s \ge 3$. Another approach was taken by Danchin \cite{D} in 2001 using the Besov spaces $B^s_{p,r}(S^1)$ with $s>\max\{1+1/p,3/2\}$, $1 \le p \le \infty$, and $1 \le r<\infty$. For the nonperiodic case, local well-posedness was proved for initial data in $H^s(\RR)$ with $s>3/2$ by Li and Olver \cite{LiO} and Rodriguez-Blanco \cite{R}.

In 2002, Misiolek \cite{Mis2002} proved the local well-posedness of the periodic CH equation in the space of continuously differentiable functions on $S^1$, by viewing the equation as an ODE in a Banach space using the geometric interpretation. In this paper, we want to establish an analogous $C^1$ result for the non-periodic problem using a similar technique. The main difference when we consider the non-periodic case is that the domain is not compact, so we must consider $C^1$ diffeomorphisms with an appropriate decaying condition. We study this diffeomorphism group and show that it is a topological group, so that operations of composition and inversion are continuous.

It is interesting to compare the present work with the recent paper by Linares at el. \cite{LPS}. There, the authors assume slightly weaker hypotheses: their initial data $u_0$ is in $H^1 \cap W^{1,\infty}$ rather than $H^1 \cap C^1$, which allows them to include peakon solutions of the form $u(t,x)=e^{-\abs{x-ct}}$. Using different methods, they obtain local existence and uniqueness in this space, but significantly they do not obtain continuous dependence on the initial data (and in fact an explicit example shows it is \emph{false} in that context). This demonstrates that the group of piecewise $C^1$ diffeomorphisms on $\RR$ does \emph{not} form a topological group, since the continuity properties of the inversion and composition are the primary tool to make our proofs work.

Note that even when the solution operator is continuous, it is not uniformly continuous even in very strong Sobolev topologies; see for example Himonas-Kenig-Misio{\l}ek~\cite{HKM2010}. This is a consequence of the failure of the group operations to be uniformly continuous, as we will see. 

In section 3, we use the Lagrangian approach for the local well-posedness of the CH equation. That is, we will write the equation entirely in terms of the flow $\eta$ of particle trajectories, which turns the equation into an ODE on the open subset of a Banach space. We will show that the resulting vector field is locally Lipschitz, so that we can apply the Picard Theorem for ODEs in Banach space (see Lang \cite{L}). Topological group properties will then ensure that the solution depends continuously on the initial data, which completes the proof of the local well-posedness.

\section{The group of $C^1 \cap H^1$ diffeomorphisms}

\begin{definition}
We denote by $\mathcal{D}(\RR)$ the set of maps $\eta:\RR \to \RR$ satisfying the following
conditions
\begin{enumerate}
\item $\eta(x)$ is a $C^1$ function with a bound $a \le \eta'(x) \le b$ on $\RR$ for some $0<a<b$,
\item $\int_{\RR}\abs{\eta(x)-x}^2 dx<\infty$ and $\int_{\RR}\abs{\eta'(x)-1}^2\;dx<\infty$,
\item ${\displaystyle \lim_{|x| \to \infty}\eta'(x)=1}$.
\end{enumerate}
Similarly, denote by $\mathcal{V}_1(\RR)$ the set of maps $u:\RR \to \RR$ satisfying the
conditions
\begin{enumerate}
\item $u(x)$ is a $C^1$ function with a bound $|u'(x)| \le M$ on $\RR$ for some $M>0$,
\item $\int_{\RR}\abs{u(x)}^2 dx<\infty$ and $\int_{\RR}\abs{u'(x)}^2\;dx<\infty$,
\item ${\displaystyle \lim_{|x| \to \infty}u'(x)=0}$.\end{enumerate}
\end{definition}

The topology of $\mathcal{V}_1(\RR)$ is generated by the following norm:
\begin{equation}\label{norm}
\norm{u}_{1,1}=\norm{u}_{C^1}+\norm{u}_{H^1}=\sup_{x \in \RR}\abs{u(x)}+\sup_{x \in
\RR}\abs{u'(x)}+\sqrt{\int_{\mathbb{R}}u(x)^2 dx+\int_{\RR}u'(x)^2 dx},
\end{equation}
and the corresponding distance on the space $\mathcal{D}(\RR)$ is given by
\begin{equation}\label{distance}
D(\eta,\xi)=\norm{\eta-\xi}_{1,1}.
\end{equation}

The conditions in the above definition have useful consequences, which we list here; the proof is straightforward.
\begin{lemma}\label{limitatinfty}
Suppose that $f$ is a $C^1$ function with $|f'(x)| \le M$ for all $x$ and $\int_{\RR}f(x)^2 dx<\infty$. Then ${\displaystyle \lim_{|x| \to \infty}f(x)=0}$. If, in addition, we have ${\displaystyle \lim_{|x| \to \infty}f'(x)=0}$, then $f'$ is uniformly continuous.
\end{lemma}
From this lemma, we conclude another decaying property: ${\displaystyle \lim_{|x| \to \infty}\abs{\eta(x)-x}=0}$ and ${\displaystyle \lim_{x \to \infty}u(x)=0}$. Furthermore we know that $\eta'$ and $u'$ are uniformly continuous. This uniform continuity of the derivatives is a very useful tool for many estimates in this paper, and in fact it is not hard to check that uniform continuity of the derivative together with an $H^1$ bound implies the decay properties above.

\begin{remark} We can easily check that $\mathcal{D}(\RR)$ is a topological manifold: given
$\eta \in \mathcal{D}(\RR)$, let $v(x):=\eta(x)-x$; then $v$ and $v'$ are bounded continuous
functions approaching zero asymptotically, satisfying $v'(x)>-1$ for all $x$ and
$\norm{v}_{L^2}<\infty$. Conversely, any such $v$ gives $\eta=\id+v$ in $\mathcal{D}(\RR)$, so
that the map $\eta \mapsto v$ is a bijection. The image of this bijection is the set $\{v \in
\mathcal{V}_1(\RR)\;:\;\Phi(v)>-1\}$ where ${\displaystyle \Phi(v):=\min_{x \in \RR}v'(x)}$,
which is obviously a continuous function in the topology \eqref{distance}; hence, we have a map to an
open convex subset of a Banach space $\mathcal{V}_1(\RR)$. This makes $\mathcal{D}(\RR)$
a manifold with just one chart.
\end{remark}

Next, we show that $\mathcal{D}(\RR)$ is a group. We know that $C^1$ diffeomorphisms themselves
form a group, but we need to check that it is closed under the additional conditions we have imposed.
After that, we will verify that it is a topological group as well, which is a bit more involved.

\begin{proposition}
The set $\mathcal{D}(\RR)$ is a group.
\end{proposition}
\begin{proof}
First we show that $\mathcal{D}(\RR)$ contains all inverses. If $\eta \in \mathcal{D}(\RR)$, then
$\xi:=\eta^{-1}$ exists for all $x$ by the inverse function theorem on $\RR$, and it is in $C^1$.
Furthermore, since $\xi'(x)=\eta'(\xi(x))^{-1}$, we see that bounds $a \le \eta'(x) \le b$ imply the
bounds $b^{-1} \le \xi'(x) \le a^{-1}$. Also, we can easily see that $\xi'(x) \to 1$ as $x \to \infty$,
since $\eta'(x) \to 1$. The fact that $\xi$ has finite $L^2$ distance from the identity comes from
$$\int_{\RR}\abs{\xi(x)-x}^2\;dx=\int_{\RR}\abs{\xi(\eta(y))-\eta(y)}^2 \eta'(y)\;dy \le
b\int_{\RR}\abs{\eta(y)-y}^2\;dy<\infty,$$
using the change of variables formula with $x=\eta(y)$. Similarly, we can check
$$\int_{\RR} \abs{\xi'(x)-1}^2\;dx \le \int_{\RR}\abs{\frac{1}{\eta'(y)}-1}^2\eta'(y)\;dy \le \frac{1}{a}\int_{\RR}\abs{\eta'(y)-1}^2\;dy<\infty.$$

Next, we show that $\mathcal{D}(\RR)$ contains compositions. Let $\eta, \phi \in \mathcal{D}(\RR)$ with the
bounds $0<a\le \eta'(x) \le b$ and $0<c \le \phi'(x) \le d$. Since
$(\phi \circ \eta)'(x)=\phi'(\eta(x))\eta'(x)$, the limit $(\phi \circ \eta)'(x) \to 1$ is obvious, and we also have the obvious bounds $ac \le (\phi \circ
\eta)'(x) \le bd$. Lastly, we can check that
$$\int_{\RR}\abs{\phi(\eta(y))-y}^2\;dy=\int_{\RR}\abs{\phi(x)-\xi(x)}^2 \xi'(x)\;dx \le
2a^{-1}\int_{\RR}\abs{\phi(x)-x}^2+2a^{-1}\int_{\RR}\abs{\xi(x)-x}^2\;dx<\infty,$$
and
\begin{align*}
\int_{\RR}\abs{\phi'(\eta(y))\eta'(y)-1}^2\;dy &\le 2\int_{\RR}\abs{\phi'(\eta(y))\eta'(y)-\eta'(y)}^2\;dy+2\int_{\RR}\abs{\eta'(y)-1}^2\;dy\\
&\le 2b\int_{\RR}\abs{\phi'(x)-1}^2\;dy+2\int_{\RR}\abs{\eta'(y)-1}^2\;dy<\infty.
\end{align*}
\end{proof}

We now show that $\mathcal{D}(\RR)$ is a topological group, which means that the inversion and composition maps are continuous in the distance \eqref{distance}. We prove this in the following lemmas.

\begin{lemma}\label{topgroup}
The map $\mathrm{Inv}:\eta \mapsto \eta^{-1}$ is continuous.
\end{lemma}

\begin{proof}
Let $\eta_1 \in \mathcal{D}(\RR)$ with $\mathrm{Inv}(\eta_1)=\xi_1$. We want to show that the map $\mathrm{Inv}$ is continuous at $\eta_1$. 
Let $\epsilon>0$ be given. We have to bound $\norm{\xi_1-\xi_2}_{1,1}$ in terms of $\rho:=\norm{\eta_1-\eta_2}_{1,1}$ for all $\eta_2 \in \mathcal{D}(\RR)$ with the inverse $\xi_2$. We will estimate $L^\infty$ and $L^2$ norms of $\xi_1-\xi_2$ and $\xi_1'-\xi_2'$ separately.

From the definition, we have $a_i \le
\eta_i'(x) \le b_i$ for $i=1,2$. Note that $a_2 \ge a_1-\rho$ and $b_2 \le b_1 +\rho$ which we will use later. Our goal is to get the bound of $\norm{\xi_1-\xi_2}_{1,1}$ solely in terms of quantities depending on $\eta_1$. From Lemma \ref{limitatinfty}, $\eta_1'$ is uniformly continuous and so there is a modulus of continuity $\omega_1:\RR
\to \RR$ satisfying $\omega_1(0)=\lim_{\rho \to 0} \omega_1(\rho)=0$ with $\omega_1$ increasing such
that
$$\abs{\eta_1'(x_1)-\eta_1'(x_2)}\le \omega_1(\abs{x_1-x_2}).$$
First, note that
$$\abs{\xi_1(x)-\xi_2(x)}=\abs{\xi_1(\eta_2(y))-y}=\abs{\xi_1(\eta_2(y))-\xi_1(\eta_1(y))} \le a_1^{-1}\abs{\eta_1(y)-\eta_2(y)}
\le \rho/a_1$$
for all $x$, where $x=\eta_2(y)$. So
$$\norm{\xi_1-\xi_2}_{L^\infty} \le \frac{\rho}{a_1}.$$
Next, using the same trick as above with $x=\eta_2(y)$, we write
\begin{align*}
\int_{\RR}\abs{\xi_1(x)-\xi_2(x)}^2\;dx=\int_{\RR}\abs{\xi_1(\eta_2(y))-\xi_1(\eta_1(y))}^2
\eta_2'(y)\;dy \le \frac{b_2}{a_1^2}\int_{\RR}\abs{\eta_2(y)-\eta_1(y)}^2\;dy \le
\frac{b_1+\rho}{a_1^2}\rho^2,
\end{align*}
since $b_2 \le b_1+\rho$. We conclude that
$$\norm{\xi_1-\xi_2}_{L^2} \le \frac{\sqrt{b_1+\rho}}{a_1}\rho.$$

Now we estimate $\norm{\xi_1'-\xi_2'}_{L^\infty}$. We have
\begin{align*}
\abs{\xi_1'(x)-\xi_2'(x)}&=\abs{\eta_1'(\xi_1(x))^{-1}-\eta_2'(\xi_2(x))^{-1}}\\
&=\frac{1}{\abs{\eta_1'(\xi_1(x))}\abs{\eta_2'(\xi_2(x))}}\abs{\eta_1'(\xi_1(x))-\eta_2'(\xi_2(x))}\\
&\le\frac{1}{a_1a_2}\left(\abs{\eta_1'(\xi_1(x)-\eta_1'(\xi_2(x)))}+\abs{\eta_1'(\xi_2(x))-\eta_2'(\xi_2(x))}\right)\\
&\le\frac{1}{a_1(a_1-\rho)}\left(\omega_1\left(\abs{\xi_1(x)-\xi_2(x)}\right)+\rho\right)\\
&\le\frac{1}{a_1(a_1-\rho)}\left(\omega_1(\rho/a_1)+\rho\right).
\end{align*}
Hence
$$\norm{\xi_1'-\xi_2'}_{L^\infty} \le \frac{1}{a_1(a_1-\rho)}\left(\omega_1(\rho/a_1)+\rho\right).$$

The last estimate $\norm{\xi_1'-\xi_2'}_{L^2}$ is a little bit more complicated than the previous ones. We use heavily the uniform continuity and work directly in terms of the $\epsilon>0$ given above. First, note that we can write
\begin{align*}
\norm{\xi_1'-\xi_2'}^2_{L^2}&=\int_{\RR}\left(\xi_1'(x)-\xi_2'(x)\right)^2\;dx\\
&=\int_{\RR}\left(\frac{1}{\eta_1'(\xi_1(x))}-\frac{1}{\eta_2'(\xi_2(x))}\right)^2\;dx\\
&=\int_{\RR}\frac{1}{\eta_1'(\xi_1(x))^2\eta_2'(\xi_2(x))^2}\left(\eta_1'(\xi_1(x))-\eta_2'(\xi_2(x))\right)^2\;dx\\
&\le\frac{1}{(a_1 a_2)^2}\int_{\RR}\left(\eta_1'(\xi_1(x))-\eta_2'(\xi_2(x))\right)^2\;dx\\
&\le\frac{2}{a_1^2
(a_1-\rho)^2}\left(\int_{\RR}\left(\eta_1'(\xi_2(x))-\eta_2'(\xi_2(x))\right)^2\;dx+\int_{\RR}\left(\eta_1'(\xi_1(x))-\eta_1'(\xi_2(x))\right)^2\;dx\right)\\
\end{align*}
Then for the first integral, we have
\begin{align*}
\int_{\RR}\left(\eta_1'(\xi_2(x))-\eta_2'(\xi_2(x))\right)^2\;dx=\int_{\RR}\left(\eta_1'(y)-\eta_2'(y)\right)^2
\eta_2'(y)\;dy\le b_2 \rho^2 \le (b_1+\rho)\rho^2.
\end{align*}
For the second integral, we use the fact that the function $\eta_1'(x) - 1$ is in $L^2$ from the assumption on $\mathcal{D}(\RR)$. So there exists $M>0$ such that
$$\int_{|x| > M}(\eta_1'(z)-1)^2\;dz<\epsilon^2 \frac{a_1^2(a_1-\rho)^2}{32(2b_1+\rho)}.$$
Now choose $M' = M+2\norm{\xi_1-\id}_{L^{\infty}}$; then $\lvert x\rvert > M'$ implies $\lvert \xi_1(x)\rvert > M$. In addition
if $\rho<\norm{\xi_1-\id}_{L^{\infty}}$ then $\lvert x\rvert > M'$ also implies $\lvert \xi_2(x)\rvert > M$.

Then
\begin{align*}
\int_{\RR}\left(\eta_1'(\xi_1(x))-\eta_1'(\xi_2(x))\right)^2\;dx&=\int_{|x| \le M'}\left(\eta_1'(\xi_1(x))-\eta_1'(\xi_2(x))\right)^2\;dx+\int_{|x|> M'}\left(\eta_1'(\xi_1(x))-\eta_1'(\xi_2(x))\right)^2\;dx\\
&=: (I)+(II).
\end{align*}
Then
\begin{align*}
(II) &\le 2\int_{|x|> M'}(\eta_1'(\xi_1(x))-1)^2\;dx+2\int_{|x|> M'}(\eta_1'(\xi_2(x))-1)^2\;dx\\
&\le 2\int_{|z|> M}(\eta_1'(z)-1)^2 \eta_1'(z)\;dz+2\int_{|z|> M}(\eta_1'(z)-1)^2\eta_2'(z)\;dz\\
&<\frac{\epsilon^2 a_1^2(a_1-\rho)^2}{16}.
\end{align*}
For $(I)$, we use the uniform continuity of $\eta_1'$ directly. That is we can find $\delta>0$ such that if $\abs{\xi_1(x)-\xi_2(x)}<\delta$, then
$$\abs{\eta_1'(\xi_1(x))-\eta_1'(\xi_2(x))}<\frac{\epsilon a_1(a_1-\rho)}{\sqrt{32M'}},$$
for all $x \in \mathbb{R}$. Requiring $\rho$ to be smaller than $\delta$ implies that $(I) \le \frac{\epsilon^2 a_1^2(a_1-\rho)^2}{16}$.

Combining the bounds for $(I)$ and $(II)$, we get $\int_{\RR}\left(\eta_1'(\xi_1(x))-\eta_1'(\xi_2(x))\right)^2\;dx < \frac{\epsilon^2 a_1^2(a_1-\rho)^2}{8}$, and then
$$\norm{\xi_1'-\xi_2'}_{L^2} < \frac{\sqrt{2(b_1+\rho)}}{a_1(a_1-\rho)}\rho+\frac{\epsilon}{2}.$$

Combining the four inequalities above, we see that
\begin{equation}\label{locallipschitzinverse}
\norm{\xi_1-\xi_2}_{1,1}\le\left[\frac{1}{a_1}+\frac{1}{a_1(a_1-\rho)}+\frac{\sqrt{b_1+\rho}}{a_1}+\frac{\sqrt{2(b_1+\rho)}}{a_1(a_1-\rho)}\right]\rho+\frac{\omega_1(\rho/a_1)}{a_1(a_1-\rho)}+\frac{\epsilon}{2}.
\end{equation}
Choosing $\rho$ sufficiently small, the right hand side can be made less than $\epsilon$. This proves the continuity of the inversion map.
\end{proof}

Note that in the periodic case the same proof applies, except we do not need to estimate the $L^2$ tail. On the other hand the estimate \eqref{locallipschitzinverse} still requires us to use the modulus of continuity $\omega_1$ of the fixed diffeomorphism $\eta_1$, and thus
we do not get \emph{uniform} continuity of the inversion map. This is responsible for the failure of the Camassa-Holm solution map to be uniformly continuous in the data even in spaces with higher smoothness, as mentioned in the Introduction. 

It remains to prove the continuity of the composition mapping. We will prove the following lemma which is a more general statement than the continuity of composition of diffeomorphisms. That is, we will prove that the the composition mapping of a vector $\phi \in \mathcal{V}_1(\RR)$ and a diffeomorphism $\eta \in \mathcal{D}(\RR)$ is continuous. Then the composition mapping of two diffeomorphisms will be continuous as a consequence. Furthermore, this type of composition will appear later in the argument of local well-posedness.

\begin{lemma}\label{composition1}
The map $\mathrm{Comp}^1:\mathcal{V}_1(\RR) \times \mathcal{D}(\RR) \to \mathcal{V}_1(\RR)$ given by $(\phi,\eta)\mapsto \phi\circ\eta$ is continuous.
\end{lemma}

\begin{proof}
Let $(\phi_i,\eta_i) \in \mathcal{V}_1 \times \mathcal{D}$ for $i=1,2$. Then $a_i \le
\eta_i'(x) \le b_i$ with inverses $\xi_i$ for $\eta_i$ and we
have the bounds $|\phi_i'(x)| \le C_i$ on $\RR$. In addition, there is a modulus of continuity $\omega_1$ since $\phi_1'$ is uniformly continuous as before. We claim that we can control the norm $\norm{\phi_1 \circ \eta_1-\phi_2 \circ \eta_2}_{1,1}$ in terms of $\rho:=\norm{\eta_1-\eta_2}_{1,1}$ and $\sigma:=\norm{\phi_1-\phi_2}_{1,1}$. First, note that $\abs{\phi_1(\eta_1(x))-\phi_2(\eta_2(x))} \le \abs{\phi_1(\eta_1(x))-\phi_1(\eta_2(x))}+\abs{\phi_1(\eta_2(x))-\phi_2(\eta_2(x))}$. By the Mean Value Theorem and the bound for $\phi_1'$,
$$\abs{\phi_1(\eta_1(x))-\phi_1(\eta_2(x))} \le C_1\abs{\eta_1(x)-\eta_2(x)} \le C_1 \norm{\eta_1-\eta_2}_{L^\infty}.$$
Also, $\abs{\phi_1(\eta_2(x))-\phi_2(\eta_2(x))} \le \norm{\phi_1-\phi_2}_{L^\infty}$. Hence,
$$\norm{\phi_1 \circ \eta_1-\phi_2 \circ \eta_2}_{L^\infty} \le C_1 \norm{\eta_1-\eta_2}_{L^\infty}+\norm{\phi_1-\phi_2}_{L^\infty} \le C_1\rho+\sigma.$$

Next,
\begin{align*}
\abs{(\phi_1 \circ \eta_1)'-(\phi_2 \circ \eta_2)'} &=\abs{\phi_1'(\eta_1(x))\eta_1'(x)-\phi_2'(\eta_2(x))\eta_2'(x)}\\
&\le\abs{\phi_1'(\eta_1(x))\eta_1'(x)-\phi_1'(\eta_1(x))\eta_2'(x)}\\
&\;\;\;\;\;+\abs{\phi_1'(\eta_1(x))\eta_2'(x)-\phi_1'(\eta_2(x))\eta_2'(x)}\\
&\;\;\;\;\;+\abs{\phi_1'(\eta_2(x))\eta_2'(x)-\phi_2'(\eta_2(x))\eta_2'(x)}\\
&=: (I)+(II)+(III).
\end{align*}
Then
\begin{align*}
(I) &\le C_1\abs{\eta_1'(x)-\eta_2'(x)}\le C_1 \norm{\eta_1'-\eta_2'}_{L^\infty},\\
(II) &\le b_2\abs{\phi_1'(\eta_1(x))-\phi_1'(\eta_2(x))}\le b_2 \omega_1(|\eta_1(x)-\eta_2(x)|)\le b_2 \omega_1(\norm{\eta_1-\eta_2}_{L^\infty}),\\
(III) &\le b_2\abs{\phi_1'(\eta_2(x))-\phi_2'(\eta_2(x))}\le b_2 \norm{\phi_1'-\phi_2'}_{L^\infty}.
\end{align*}
Hence,
$$\norm{(\phi_1 \circ \eta_1)'-(\phi_2 \circ \eta_2)'}_{L^\infty} \le (C_1+b_2) \rho+b_2 \omega_1(\rho).$$

For $\norm{\phi_1 \circ \eta_1-\phi_2 \circ \eta_2}_{L^2}$, we have
$$\norm{\phi_1 \circ \eta_1-\phi_2 \circ \eta_2}_{L^2} \le \norm{\phi_1 \circ \eta_1-\phi_2 \circ \eta_1}_{L^2}+\norm{\phi_2 \circ \eta_1-\phi_2 \circ \eta_2}_{L^2},$$
by the triangle inequality. Then
\begin{align*}
\norm{\phi_1 \circ \eta_1-\phi_2 \circ \eta_1}_{L^2}^2=\int_{\RR}\left(\phi_1(\eta_1(x))-\phi_2(\eta_1(x))\right)^2\;dx=\int_{\RR}\left(\phi_1(y)-\phi_2(y)\right)^2\xi_1'(y)\;dy\le\frac{1}{a_1}\norm{\phi_1-\phi_2}_{L^2}^2,
\end{align*}
and
\begin{align*}
\norm{\phi_2 \circ \eta_1-\phi_2 \circ \eta_2}_{L^2}^2=\int_{\RR}\left(\phi_2(\eta_1(x))-\phi_2(\eta_2(x))\right)^2\;dx \le C_2^2\int_{\RR}(\eta_1(x)-\eta_2(x))^2\;dx=C_2^2 \norm{\eta_1-\eta_2}_{L^2}^2.
\end{align*}
Hence,
$$\norm{\phi_1 \circ \eta_1-\phi_2 \circ \eta_2}_{L^2} \le \frac{1}{\sqrt{a_1}}\sigma+C_2\rho.$$

Lastly, we estimate $\norm{(\phi_1 \circ \eta_1)'-(\phi_2 \circ \eta_2)'}_{L^2}$. Note that
\begin{align*}
\norm{\phi_1'(\eta_1(x))\eta_1'(x)-\phi_2'(\eta_2(x))\eta_2'(x)}_{L^2}&\le\norm{\phi_1'(\eta_1(x))\eta_1'(x)-\phi_1'(\eta_1(x))\eta_2'(x)}_{L^2}\\
&\;\;\;\;\;+\norm{\phi_1'(\eta_1(x))\eta_2'(x)-\phi_1'(\eta_2(x))\eta_2'(x)}_{L^2}\\
&\;\;\;\;\;+\norm{\phi_1'(\eta_2(x))\eta_2'(x)-\phi_2'(\eta_2(x))\eta_2'(x)}_{L^2}\\
&=:(I)+(II)+(III).
\end{align*}
By using a similar technique, we can find $(I) \le C_1 \rho$\; and $(III) \le \frac{b_2}{\sqrt{a_2}}\sigma$. For $(II)$, we use the method of splitting the integral into two parts and then using the uniform continuity as in the proof of Lemma 6. Here, the situation is slightly simpler since $\phi_1'$ itself is in $L^2$. So for any $\epsilon>0$ we have
$$\norm{(\phi_1 \circ \eta_1)'-(\phi_2 \circ \eta_2)'}_{L^2}<\frac{\epsilon}{2}$$
if $\rho$ and $\sigma$ are sufficiently small.

Now combining the four inequalities, we obtain
\begin{align*}
\norm{\eta_1 \circ \phi_1-\eta_2 \circ
\phi_2}_{1,1} \le \left(1+\frac{1}{\sqrt{a_1}}+b_1+\rho\right)\sigma+\left(3C_1+\sigma\right)\rho+b\omega_1(\rho)+\frac{\epsilon}{2}.
\end{align*}
Then we can make the right side as small as we want by choosing
$\sigma$ and $\rho$ sufficiently small; hence $\eta \circ \phi$ is continuous as a function of
$(\eta,\phi)$ in the product metric.
\end{proof}

Again we note that this map is not uniformly continuous. 

\begin{corollary}\label{composition2}
The map $\mathrm{Comp}^2:\mathcal{D}(\RR) \times \mathcal{D}(\RR) \to \mathcal{D}(\RR)$ is continuous.
\end{corollary}
\begin{proof}
Let $(\xi,\eta) \in \mathcal{D} \times \mathcal{D}$. Then $\phi:=\xi-\mathrm{Id}$ is an element of $\mathcal{V}_1$. From Lemma \ref{composition1}, the mapping $(\phi, \eta)\mapsto \phi \circ \eta$ is continuous. Note that
$$\phi \circ \eta=(\xi-\mathrm{Id})\circ \eta=\xi \circ \eta-\eta,$$
and so we can identify $\mathrm{Comp}^2=\mathrm{Comp}^1+\mathrm{Proj}_2$, where $\mathrm{Proj}_2$ is the projection map onto the second component. Since $\mathrm{Comp}^2$ is the sum of two continuous mappings, it is continuous.
\end{proof}

Therefore, $\mathcal{D}(\RR)$ is a topological group.

\section{Local existence}
It is convenient to rewrite the Camassa-Holm equation in its equivalent form
\begin{equation}\label{CH2}
u_t+uu_x=-\partial_x(1-\partial_x^2)^{-1}\left(u^2+\frac{1}{2}u_x^2\right).
\end{equation}
Note that the operator $(1-\partial_x^2)$ is invertible since we can can check, for a bounded function $g$, that
\begin{equation}\label{lambdainverse}
f(x)-f''(x)=g(x)\;\text{and}\lim_{|x|\to \infty}f(x)=0 \implies
f(x)=\frac{1}{2}\int_{\RR}e^{-|x-y|}g(y)dy.
\end{equation}

Now, consider the Lagrangian flow equation
\begin{equation}\label{floweq}
\frac{\partial \eta}{\partial t}(t,x)=u\left(t,\eta(t,x)\right).
\end{equation}

Since $u$ is in $\mathcal{V}_1(\RR)$, the function $\phi:=u^2+\frac{1}{2}u_x^2$ in the parenthesis of \eqref{CH2} is continuous, integrable, and decaying to zero at infinity. Then this implies that the function $\phi$ is also square integrable. So we will define the collection of such functions $\phi$ as $\mathcal{V}_0$ below and proceed with the local existence in the space $\mathcal{V}_0(\RR)$. Then we will realize $\phi$ as the function $u^2+\frac{1}{2}u_x^2$ to finish the proof of local well-posedness.

\begin{definition}
We denote by $\mathcal{V}_0(\RR)$ the set of maps $u:\RR \to \RR$ satisfying the
conditions
\begin{enumerate}
\item $u(x)$ is a continuous function with ${\displaystyle \lim_{|x| \to \infty}u(x)=0}$, and
\item $\int_{\RR}\abs{u(x)}^2 dx<\infty$.
\end{enumerate}
\end{definition}

So $\mathcal{V}_0=C^0_0 \cap L^2$, where $C^0_0$ denotes the space of continuous functions that decay to zero at infinity. Hence the norm in $\mathcal{V}_0$ is the sum of $L^\infty$ and $L^2$ norms. Now we can write the equation \eqref{CH2} in terms of $\eta$.

\begin{proposition}
The equation \eqref{CH2} can be rewritten in terms of the flow $\eta$ as
\begin{equation}\label{lagrangian}
\eta_{tt}=\mathcal{L}_\eta\left(\eta_t^2+\frac{\eta_{tx}^2}{2\eta_x^2}\right),
\end{equation}
where $\mathcal{L}$ is defined by
\begin{equation}\label{Lfunction}
\mathcal{L}_\eta(\phi)=\mathcal{L}(\phi \circ \eta^{-1}) \circ
\eta,\;\text{and}\;\mathcal{L}=\partial_x(1-\partial_x^2)^{-1}\;\text{for any function $\phi$.}
\end{equation}
\end{proposition}

\begin{proof}
We differentiate the equation \eqref{floweq} with respect to $t$ and use the chain rule. We obtain
$$\frac{\partial^2 \eta}{\partial t^2}(t,x)=\frac{\partial u}{\partial t}(t,\eta(t,x))+\frac{\partial u}{\partial x}(t,\eta(t,x))\frac{\partial \eta}{\partial x}(t,x)=\frac{\partial u}{\partial t}(t,\eta(t,x))+u(t,\eta(t,x))\frac{\partial u}{\partial x}(t,\eta(t,x)),$$
which is the left hand side of equation \eqref{CH2} composed with $\eta$. Also,
$$\frac{\partial^2 \eta}{\partial t \partial x}(t,x)=\frac{\partial u}{\partial x}(t,\eta(t,x))\frac{\partial \eta}{\partial x}(t,x),$$
so we get $u_x \circ \eta=\frac{\eta_{tx}}{\eta_x}$. Now, if $p-p_{xx}=u^2+\frac{1}{2}u_x^2$, then the equation \eqref{CH2} becomes
$$\eta_{tt}=-p_x \circ \eta.$$
After taking the inverse flow $\eta^{-1}$ to set the variables at the right place and writing everything in terms of operators, we get the equation \eqref{lagrangian}.
\end{proof}

The first thing we notice is that as a function of $(\eta,V)=(\eta,\eta_t)$, the right side of equation
\eqref{lagrangian} does not lose derivatives. If $\eta$ and $V$ are both $C^1$, then the term inside
parentheses is continuous, while $\mathcal{L}$ gains a derivative so that if $\phi \in C^0$, then
$\mathcal{L}_\eta(\phi) \in C^1$. Hence, the equation \eqref{lagrangian} becomes a first-order equation on an
open subset of the Banach space $\mathcal{D}(\RR) \times \mathcal{V}_1(\RR)$. We claim
that the right side function is $C^1$ in the $(\eta,V)$ variables. We begin with the following lemma.

\begin{lemma}\label{boundedlinear}
For each $\eta \in \mathcal{D}(\RR)$, $\mathcal{L}_\eta:\mathcal{V}_0(\RR) \to \mathcal{V}_1(\RR)$ is a bounded linear operator.
\end{lemma}

\begin{proof}
By formula \eqref{lambdainverse} we can write \eqref{lagrangian} as
\begin{equation}\label{L}
\mathcal{L}(\phi)(x)=-\frac{1}{2}e^{-x}\int_{-\infty}^x e^z \phi(z)dz+\frac{1}{2}e^x \int_{x}^\infty
e^{-z}\phi(z)\;dz
\end{equation}
and thus
$$\mathcal{L}_\eta (\phi)(x)=-\frac{1}{2}e^{-\eta(x)}\int_{-\infty}^x
e^{\eta(y)}\phi(y)\eta'(y)\;dy+\frac{1}{2}e^{\eta(x)}\int_{x}^{\infty}e^{-\eta(y)}\phi(y)\eta'(y)\;dy=:f(x).$$
We need to show that $\norm{f}_{1,1}$ is bounded by $\norm{\phi}_{L^\infty}$ and $\norm{\phi}_{H^1}$. Since $\eta \in \mathcal{D}(\RR)$ and $\phi \in \mathcal{V}_0(\RR)$, we have $a \le \eta' \le b$ and $-M\le \phi \le M$ for some constants $a,b,M>0$. Since $\eta$ is
strictly increasing,
$$\left\{
  \begin{array}{l l}
     -\eta(x)+\eta(y)\le -a(x-y) & \quad \text{if $-\infty<y<x$,}\\
      \eta(x)-\eta(y)\le -a(y-x) & \quad \text{if $x<y<\infty$.}
   \end{array} \right.$$
Hence,
\begin{align*}
\abs{f(x)}&\le\frac{Mb}{2}\left(\int_{-\infty}^x e^{-a(x-y)}\;dy+\int_x^{\infty} e^{-a(y-x)}\;dy\right)=\frac{Mb}{a}.
\end{align*}
So $\norm{f}_{L^\infty} \le \frac{b}{a}\norm{\phi}_{L^\infty}$. Similarly, we can find $\norm{f'}_{L^\infty} \le \left(\frac{b^2}{a}+b\right)\norm{\phi}_{L^\infty}$.

Note that by definition \eqref{Lfunction} we can write $f(x)=q(\eta(x))$ where
$q$ satisfies the differential equation $q(x)-q''(x)=\psi'(x)$, with $\psi=\phi\circ\xi$ and $\xi=\eta^{-1}$. Multiplying by $q(x)$ on both sides and
integrating on $\mathbb{R}$, we get
$$\int_{-\infty}^\infty q(z)^2+q'(z)^2\;dz=\int_{-\infty}^\infty \psi'(z)q(z)\;dz = -\int_{-\infty}^{\infty} \psi(z)q'(z)\;dz.$$
This implies that
$\left\Vert q \right\Vert_{H^1}^2 \le \norm{\psi}_{L^2} \norm{q}_{H^1},$
and thus that
$$ \norm{q}_{H^1} \le \norm{\psi}_{L^2}.$$
Since in addition we have
$$\norm{\psi}_{L^2}^2=\int_{\RR}\abs{\phi(\xi(z))}^2\;dz=\int_{\RR}\abs{\phi(y)}^2\eta'(y)\;dy \le b\norm{\phi}_{L^2}^2<\infty$$
we obtain $\norm{q}_{H^1} \le \sqrt{b} \norm{\phi}_{L^2}$. Note that
$$\norm{q \circ \eta}_{L^2}^2
=\int_{\RR}\abs{q(\eta(y))}^2\;dy=\int_{\RR}\abs{q(z)}^2\xi'(z)\;dz\le\frac{1}{a}\int_{\RR}\abs{q(z)}^2\;dz=\frac{1}{a}\norm{q}_{L^2}^2.$$
Similarly, $\norm{(q\circ \eta)'}_{L^2}^2 \le b\norm{q'}_{L^2}^2$. Hence,
$$\norm{f}_{H^1}=\norm{q \circ \eta}_{H^1} \le \left(\frac{1}{\sqrt{a}}+\sqrt{b}\right)\norm{q}_{H^1}\le\left(\frac{\sqrt{b}}{\sqrt{a}}+b\right)\norm{\phi}_{L^2}.$$
By combining all inequalities, we get
$$\norm{f}_{1,1}=\norm{f}_{C^1}+\norm{f}_{H^1} \le \left(\frac{b+b^2}{a}+b\right)\norm{\phi}_{L^\infty}+\left(\frac{\sqrt{b}}{\sqrt{a}}+b\right)\norm{\phi}_{L^2},$$
and this proves the continuity.
\end{proof}

\begin{theorem}\label{C1}
The function $F(\phi,\eta)=\mathcal{L}_{\eta}(\phi)$, where the operator $\mathcal{L}_\eta$ is defined by formula \eqref{Lfunction}, is a continuously differentiable function
from $\mathcal{V}_0(\RR) \times 
\mathcal{D}(\RR)$ to $\mathcal{V}_1(\RR)$. 
\end{theorem}

\begin{proof}
Note that $F$ is smooth as a function of $\phi$ since it is linear with respect to $\phi$ and bounded in $\phi$ from Lemma \ref{boundedlinear}. So we only have to show the continuous differentiability of $F$ with respect to $\eta$. We can compute
\begin{align}
\partial_\eta F(\phi,\eta)(\rho)&=\frac{d}{d\epsilon}\Big|_{\epsilon=0}F(\phi,\eta+\epsilon \rho)\\
&=\frac{1}{2}\int_{-\infty}^x
e^{-\eta(x)+\eta(y)}\phi(y)\left(\rho(x)\eta'(y)-\rho(y)\eta'(y)-\rho'(y)\right)\;dy\\
&\qquad\qquad +\frac{1}{2}\int_{x}^{\infty}e^{\eta(x)-\eta(y)}\phi(y)\left(\rho(x)\eta'(y)-\rho(y)\eta'(y)+\rho'(y)\right)\;dy.\label{Fderivative}
\end{align}

We first want to show that $G(x):=\partial_\eta F(\phi,\eta)(\rho)(x)$ is a function in the correct target
space $\mathcal{V}_1(\RR)$. So we must check the following:

\begin{equation*}
\begin{aligned}
&\max_{x \in \RR} |G(x)|<\infty,\;\;\\
&\max_{x \in \RR} |G'(x)|<\infty,\;\;
\end{aligned}
\begin{aligned}
&\lim_{|x| \to \infty}|G(x)|=0,\;\;\\
&\lim_{|x| \to \infty}|G'(x)|=0.\;\;
\end{aligned}
\begin{aligned}
&\int_{\RR} |G(x)|^2\;dx <\infty,\\
&\int_{\RR} |G'(x)|^2\;dx <\infty,
\end{aligned}
\end{equation*}

Observe that the two conditions in the first column follow from the two conditions in the second column. Note that $\rho \in T_\eta \mathcal{D}(\RR)=\mathcal{V}_1(\RR)$, so we have $-c \le \rho \le c$ and $-N \le \rho' \le N$ for some constants $c,N>0$.

\begin{enumerate}
\item ${\displaystyle \lim_{x \to \infty}G(x)=0}$ and ${\displaystyle \lim_{x \to \infty}G'(x)=0}$

\vspace{0.2cm}
Denote $G(x) = (I)+(II)$ in terms of the two integrals in \eqref{Fderivative}.
By using l'Hopital's Rule, we have
\begin{align*}
\lim_{x \to \infty}(I)&=\lim_{x \to \infty}\frac{\rho(x)\int_{-\infty}^x
e^{\eta(y)}\phi(y)\eta'(y)\;dy-\int_{-\infty}^x e^{\eta(y)}\phi(y)\rho(y)\eta'(y)\;dy-\int_{-\infty}^x
e^{\eta(y)}\phi(y)\rho'(y)\;dy}{2e^{\eta(x)}}\\
&=\lim_{x \to \infty}\frac{\rho'(x)}{\eta'(x)}\frac{\int_{-\infty}^x
e^{\eta(y)}\phi(y)\eta'(y)\;dy}{2e^{\eta(x)}}-\frac{\phi(x)\rho'(x)}{2\eta'(x)}\\
&=\lim_{x \to \infty}\frac{\rho'(x)}{\eta'(x)}\frac{\int_{-\infty}^x
e^{\eta(y)}\phi(y)\eta'(y)\;dy}{2e^{\eta(x)}},\;\left(\because \lim_{x \to \infty}\rho'(x)=0\right)
\end{align*}
Note that
\begin{align*}
\lim_{x \to \infty}\frac{\int_{-\infty}^x e^{\eta(y)}\phi(y)\eta'(y)\;dy}{2e^{\eta(x)}}&=\lim_{x \to
\infty}\frac{e^{\eta(x)}\phi(x)\eta'(x)}{2e^{\eta(x)}\eta'(x)}=\lim_{x \to \infty}\frac{\phi(x)}{2}=0.
\end{align*}
Hence, $\lim_{x \to \infty}(I)=0$. For $(II)$, we get another limit of indeterminate form $\frac{0}{0}$ and the computation is essentially the same. So $\lim_{x \to \infty}(II)=0$ and we get $\lim_{x \to \infty}G(x)=0$. The proofs for $\lim_{x \to -\infty} G(x)=0$ and $\lim_{|x| \to \infty}G'(x)=0$ are similar.

\item ${\displaystyle \int_{\RR} |G(x)|^2+|G'(x)|^2\;dx <\infty}$

\vspace{0.2cm}
Note that
\begin{align*}
G(x)&:=\frac{1}{2}\int_{-\infty}^x
e^{-\eta(x)+\eta(y)}\phi(y)\left(\rho(x)\eta'(y)-\rho(y)\eta'(y)-\rho'(y)\right)\;dy\\
&\;\;\;\;\;+\frac{1}{2}\int_{x}^{\infty}e^{\eta(x)-\eta(y)}\phi(y)\left(\rho(x)\eta'(y)-\rho(y)\eta'(y)+\rho'(y)\right)\;dy\\
&=\left(\frac{1}{2}\rho(x)\int_{-\infty}^x
e^{-\eta(x)+\eta(y)}\phi(y)\eta'(y)\;dy+\frac{1}{2}\rho(x)\int_x^\infty
e^{\eta(x)-\eta(y)}\phi(y)\eta'(y)\;dy\right)\\
&\;\;\;\;\;+\left(-\frac{1}{2}\int_{-\infty}^x
e^{-\eta(x)+\eta(y)}\phi(y)\rho(y)\eta'(y)\;dy-\frac{1}{2}\int_x^\infty
e^{\eta(x)-\eta(y)}\phi(y)\rho(y)\eta'(y)\;dy\right)\\
&\;\;\;\;\;+\left(-\frac{1}{2}\int_{-\infty}^x e^{-\eta(x)+\eta(y)}\phi(y)\rho'(y)\;dy+\frac{1}{2}\int_x^\infty
e^{\eta(x)-\eta(y)}\phi(y)\rho'(y)\;dy\right)\\
&=:h_1(x)+h_2(x)+h_3(x).
\end{align*}

Then we use the same method as in Lemma \ref{boundedlinear}. We can identify
$$\left\{
  \begin{array}{l l}
h_1(x)=\rho(x)q_1(\eta(x)),&\quad q_1(x)-q_1''(x)=\psi(x),\\
h_2(x)=q_2(\eta(x)),&\quad q_2(x)-q_2''(x)=-\psi(x)\rho(\xi(x)),\\
h_3(x)=q_3(\eta(x)),&\quad q_3(x)-q_3''(x)=[\psi(x)\rho'(\xi(x))\xi'(x)]',
\end{array} \right.$$

where $\psi(x)=\phi(\xi(x))$ and $\xi=\eta^{-1}$ as before. Then we have

\begin{align*}
\norm{q_1}_{H^1}&\le \norm{\psi}_{L^2},\\
\norm{q_2}_{H^1}&\le \norm{\psi (\rho \circ \xi)}_{L^2},\\
\norm{q_3}_{H^1}&\le \norm{\psi (\rho' \circ \xi) \xi'}_{L^2},
\end{align*}
where we can check that each right side is finite.

Note that since $1=s>\frac{n}{2}=\frac{1}{2}$, the Sobolev space $H^1$ is an algebra under pointwise multiplication (see Lemma 2.7 in \cite{IKT}.) Then,
$$\norm{h_1}_{H^1} \le K\norm{\rho}_{H^1}\norm{q \circ \eta}_{H^1},$$
for some constant $K$. We have
$$\norm{q \circ \eta}_{L^2}^2
=\int_{\RR}q(\eta(y))^2\;dy=\int_{\RR}q(z)^2\xi'(z)\;dz\le\frac{1}{a}\int_{\RR}q(z)^2\;dz<\infty.$$
Similarly, $\norm{(q \circ \eta)'}_{L^2}<\infty$ and we conclude that $\norm{h_1}_{H^1}<\infty$. We can show $\norm{h_2}_{H^1}<\infty$ by using the same method. The proof for $\norm{h_3}_{H^1}<\infty$ is similar to that of Lemma \ref{boundedlinear}. This completes the proof for $G \in \mathcal{V}_1(\RR)$.

It remains to show that $G$ is continuous with respect to $\eta$ to complete the proof of continuous differentiability of $F$ with respect to $\eta$. We will show that the maps $q_i$ in the proof of this Theorem, are continuous in $\eta$. The idea is to identify $q_i$ as the composition of continuous operations. We prove this in the following Lemmas. We first check that the composition map that appears in $q_i$ is continuous.
\end{enumerate}
\end{proof}

\begin{lemma}\label{comp3}
The composition mapping $\mathrm{Comp}^3:\mathcal{V}_0(\RR) \times \mathcal{D}(\RR) \to \mathcal{V}_0(\RR)$ is continuous.
\end{lemma}
\begin{proof}
The proof is analogous to that of Lemma \ref{composition1}.

\bigskip

\end{proof}

\begin{lemma}
The three maps $q_i:\mathcal{D}(\RR) \to \mathcal{V}_1(\RR)$, which are defined as above, are continuous with respect to $\eta$.
\end{lemma}
\begin{proof}
First, we can identify $q_i$ as following:
\begin{align*}
q_1&:= (1-\partial_x^2)^{-1}(\phi \circ \eta^{-1}),\\
q_2&:= -(1-\partial_x^2)^{-1}[(\phi \circ \eta^{-1})(\rho \circ \eta^{-1})],\\
q_3&:= -\partial_x(1-\partial_x^2)^{-1}[(\phi\circ\eta^{-1})((\partial_x\rho) \circ \eta^{-1})\partial_x \eta^{-1}],
\end{align*}
We will prove the continuity of $q_3$, since the proofs for $q_1$ and $q_2$ are much easier.

In the definition of $q_3$, let $r(\eta):=(\phi\circ\eta^{-1})((\partial_x\rho) \circ \eta^{-1})\partial_x \eta^{-1}$, the function inside the square bracket. Then the map $r:\mathcal{D} \to \mathcal{V}_0$ is continuous with respect to $\eta$. This is because the following maps are continuous:
\begin{itemize}
\item $\mathrm{Inv}:\eta \mapsto \eta^{-1}$ ($\because$ Lemma \ref{topgroup}),
\item $\mathrm{Comp}^3:\eta \mapsto \phi \circ \eta^{-1}$ ($\because$ Lemma \ref{comp3}),
\item $\partial_x:\rho \mapsto \partial_x\rho$, differentiation with respect to $x$,
\item multiplication of all three continuous maps.
\end{itemize}
Note that $r(\eta) \in \mathcal{V}_0$ since each term in the product are functions in $C^0_0$ and
$$\norm{r(\eta)}_{L^2}^2 \le \int_{\RR}\abs{(\phi\circ\eta^{-1})((\partial_x\rho) \circ \eta^{-1})\partial_x \eta^{-1}}^2\;dx\le \norm{\phi \circ \eta}_{L^\infty}^2 \norm{\partial_x \eta^{-1}}_{L^\infty}^2 \norm{(\partial_x \rho) \circ \eta^{-1}}_{L^2}^2 <\infty.$$
Lastly, from the explicit formula \eqref{L} we have
$$q_3=-\partial_x(1-\partial_x^2)^{-1}(r)=\frac{1}{2}e^{-x}\int_{-\infty}^x e^z r(z)dz-\frac{1}{2}e^x \int_{x}^\infty
e^{-z}r(z)\;dz.$$
Then we can check that
$$\norm{q_3}_{1,1}=\norm{q_3}_{C^1}+\norm{q_3}_{H^1}\le 2\norm{r}_{L^\infty}+\norm{r}_{L^2},$$
by doing essentially the same estimates in Lemma \ref{boundedlinear}. Hence, $q_3$ is continuous.
\end{proof}

\begin{corollary}
The map $G$ is continuous with respect to $\eta$.
\end{corollary}
\begin{proof}
Recall that $G=h_1+h_2+h_3$. The composition $q_i \circ \eta$ is continuous by Theorem \ref{composition1}, so $h_1$, $h_2$, and $h_3$ are continuous. Since the multiplication map is continuous, $h_1$ is continuous, and so the sum of three maps is continuous.
\end{proof}

Then by the following existence and uniqueness theorem for ODEs in Banach space, we get the local existence and uniqueness of the solution of $\eta$ of the Lagrangian equation \eqref{lagrangian}.

\begin{theorem}\label{Picard}\cite{L}
Let $f:J \times U \to \mathbf{E}$ be continuous, and satisfy a Lipschitz condition on $U$ uniformly with respect to $J$. Let $x_0$ be a point of $U$. Then there exists an open subinterval $J_0$ of $J$ containing $0$, and an open subset of $U$ containing $x_0$ such that $f$ has a unique flow
$$\alpha:J_0 \times U_0 \to U$$
satisfying
$$\frac{d\alpha}{dt}(t,x)=f(t,\alpha(t,x)),\;\;\;\alpha(0,x)=x.$$
We can select $J_0$ and $U_0$ such that $\alpha$ is continuous and satisfies a Lipschitz condition on $J_0 \times U_0$.
\end{theorem}

In our situation, $f=F(\phi,\eta)$, $U=\mathcal{D}(\RR)$, and $\mathbf{E}=T\mathcal{D}(\RR)$ which is the tangent bundle of $\mathcal{D}(\RR)$. The integral curve $\alpha$ corresponds to $\eta(t)$ which is a curve in $\mathcal{D}(\RR)$. Finally, we have the following theorem, which proves the local well-posedness of the Camassa-Holm equation \eqref{CH1}.

\begin{theorem}
The Cauchy problem for the Camassa-Holm equation is equivalent to the system
\begin{equation}\label{CH3}
\left\{
  \begin{array}{l l}
       \frac{d \eta}{dt}=U\\
       \frac{dU}{dt}=\mathcal{L}_\eta\left(U^2+\frac{U_x^2}{2\eta_x^2}\right)
   \end{array} \right.
\end{equation}
with initial conditions $\eta(0,\cdot)=\mathrm{Id}$, $U(0,\cdot)=u_0$. This system describes the flow of a $C^1$ vector field on $T\mathcal{D}(\RR)$ and the solution curve $(\eta,U)$ exists for some time $T>0$. Defining $u=U \circ \eta^{-1}$, we obtain a $C^1$ vector field $u:[0,T) \times \RR \to \RR$ satisfying the Camassa-Holm equation \eqref{CH1} which depends continuously on $u_0$.
\end{theorem}

\begin{proof}
Clearly, the map $(\eta,U) \mapsto U^2+\frac{U_x^2}{2\eta_x^2}:\mathcal{D}(\RR) \times \mathcal{V}_1(\RR) \to \mathcal{V}_0(\RR)$ is smooth. We have shown in the Theorem \ref{C1} that $L_\eta$ is $C^1$. Hence, the composition of these two mappings is $C^1$. Then by Theorem \ref{Picard}, there exists a time $T>0$ such that the solution curve $(\eta, U)$ of \eqref{CH3} exists on the interval $[0,T)$. This means that there is a solution mapping
\begin{equation*}\begin{array}{rll}
\Upsilon:T\mathcal{D} \!\!\!&\multicolumn{2}{l}{\to T\mathcal{D}}\\
(\eta_0, u_0)\!\!&\mapsto\ \!(\eta, u)
\end{array}\end{equation*}
Then we can construct the following composition of maps
\begin{align*}
\mathcal{V}_1  \xhookrightarrow{\;\;\;\;\;\;\iota\;\;\;\;\;\;} \;\;T\mathcal{D}  \xrightarrow{\;\;\;\;\;\Upsilon\;\;\;\;\;} T\mathcal{D} \xrightarrow{\;\;\;\;\;\;\;\;\;\;} \;&\mathcal{V}_1\\
u_0 \;\;\;\;\longmapsto\;\;(\mathrm{Id},u_0) \;\longmapsto \;(\eta,U) \;\longmapsto\;\;&U \circ \eta^{-1},
\end{align*}
where the first map is inclusion and the last map is the inversion followed by composition. Hence, we obtain the solution $u=U \circ \eta^{-1}$ of the original equation \eqref{CH1}. Continuity of $u$ follows from the fact that the above mapping is a composition of continuous maps.
\end{proof}

Although the map from $u_0$ to $u(t)$ is continuous, it is not even uniformly continuous, as mentioned earlier~\cite{HKM2010}. On the other hand
the map from $u_0$ to $\eta(t)$ is not only continuous, it is $C^1$ in both variables, which follows from the fact that the vector field in
\eqref{CH3} is $C^1$. With more work we could show that in fact the vector field and thus the solution map is $C^{\infty}$, as happens for fluids~\cite{EM1970}. Analogously we can show that Lagrangian trajectories are $C^{\infty}$ functions of time, even though the data is only spatially $C^1$. The essential feature here is that the PDE can be written in Lagrangian form in a way that does not lose derivatives, and typically this is enough to make the vector field not merely continuous but in fact $C^{\infty}$. Similar techniques should work for other Euler-Arnold equations in the $C^1$ topology, at least in one dimension.

\end{document}